\documentclass[reqno]{amsart}
\usepackage{float}
\usepackage{amsmath}
\usepackage{graphicx}
\usepackage{latexsym}
\usepackage{verbatim} 
\usepackage{color}
\usepackage{amsfonts}
\usepackage{amssymb}
\theoremstyle{plain}
\newtheorem{theorem}{Theorem}

\newtheorem{lemma}[theorem]{Lemma}
\newtheorem{proposition}[theorem]{Proposition}
\theoremstyle{definition}

\newtheorem*{remark*}{Remark}

\newcommand{\E}{\mathbf{E}}
\renewcommand{\P}{\mathbf{P}}
\newcommand{\beq}{\begin{equation}}
\newcommand{\eeq}{\end{equation}}
\newcommand{\ba}{\begin{align}}
\newcommand{\ea}{\end{align}}
\newcommand{\bas}{\begin{align*}}
\newcommand{\eas}{\end{align*}}
\newcommand{\olF}{\overline F}
\newcommand{\var}{\text{var}}
\renewcommand{\Re}{\text{Re}}
\renewcommand{\Im}{\text{Im}}

\begin{document}\title[Local limit theorem for the maximum of a random walk]
{Local limit theorem for the maximum of a random walk in the heavy-traffic regime}

\author[Kugler]{Johannes Kugler}
\address{Mathematical Institute, University of Munich,
Theresienstrasse 39, D-80333, Munich, Germany}
\email{kugler@math.lmu.de}

\date{\today }

\begin{abstract}
 Consider a family of $\Delta$-latticed aperiodic random walks $\{S^{(a)},0\le a\le a_0\}$ 
 with increments $X_i^{(a)}$ and non-positive drift $-a$. 
 Suppose that $\sup_{a\le a_0}\E[(X^{(a)})^2]<\infty$ and 
 $\sup_{a\le a_0}\E[\max\{0,X^{(a)}\}^{2+\varepsilon}]<\infty$ 
 for some $\varepsilon>0$. 
 Assume that $X^{(a)}\xrightarrow[]{w} X^{(0)}$ as $a\to 0$ and denote by 
 $M^{(a)}=\max_{k\ge 0} S_k^{(a)}$ the maximum of the random walk $S^{(a)}$.
 In this paper we provide the asymptotics of $\P(M^{(a)}=y\Delta)$ as $a\to 0$ in the case, 
 when $y\to \infty$ and $ay=O(1)$. 
 This asymptotics follows from a representation of $\P(M^{(a)}=y\Delta)$ via a geometric sum and a uniform renewal theorem, which is also proved in this paper. 
\end{abstract}

\thanks{Supported by the DFG}
\maketitle

\section{Introduction and statement of results}

Let $\{S^{(a)}, a\in[0,a_0]\}$ denote a family of random walks with drift $-a\le 0$ and increments $X_i^{(a)}$, that is,
$$
 S^{(a)}_0:=0,\ S^{(a)}_n:=\sum_{i=1}^n X_i^{(a)},\ n\ge 1.
$$
We shall assume that $X_1^{(a)},X_2^{(a)},\ldots$ are independent copies of a random variable 
$X^{(a)}$. In the case $a=0$ we write $S$, $X_i$ and $X$ instead of $S^{(0)}$, $X_i^{(0)}$ and $X^{(0)}$ respectively.
Assume that, as $a\to 0$,
\begin{equation}
 X^{(a)} \xrightarrow[]{w} X \label{assumption_weak_conv}
\end{equation}
and 
\begin{equation}
 \sup_{a\in [0,a_0]} \E [X^{(a)}]^2 <\infty \quad \text{and} \quad
 \sup_{a\in [0,a_0]} \E [(\max\{0,X^{(a)}\})^{2+\varepsilon}] <\infty \label{assumption_moment_ex}
\end{equation}
for some $a_0, \varepsilon>0$. 
If $a>0$, the random walk $S^{(a)}$ drifts to $-\infty$ and the total maximum 
$$
 M^{(a)}:=\max_{k\geq0} S^{(a)}_k
$$
is finite almost surely. 
However, as $a\to 0$, $M^{(a)}\to \infty$ in probability. 
From this fact arises the natural question how fast $M^{(a)}$ grows as $a\to 0$.
The first result concerning this question goes back to Kingman \cite{K61}, who considered the case when 
$|X|$ has an exponential moment and proved that, as $a\to 0$, 
\beq
\label{heavy_traffic}
 \P(M^{(a)}>y) \sim e^{-2ay/\sigma^2}
\eeq
for all fixed values $y\ge 0$, where $\sigma^2=\mathbf{Var}(X)$ denotes the variance of the increments in the case of zero-drift.
Prokhorov \cite{P63} extended this result to the case that the increments have finite variance.
Kingman and Prokhorov had a motivation for examining $M^{(a)}$ that comes from queueing theory: 
It is well known that a stationary distribution of the waiting time of a customer in a single-server 
first-come-first-served (GI/GI/1) queue coincides with that of the maximum of a corresponding random walk.
In the context of queueing theory, the limit $a\to 0$ means that the traffic load tends to 1. 
Thus, the question on the distribution of $M^{(a)}$ may been seen as the question on the growth rate of a stationary waiting-time 
distribution in a GI/GI/1 queue. This is one of the most important questions in queueing theory and is usually referred to as heavy-traffic analysis.

Another interesting question is whether \eqref{heavy_traffic} remains valid, if we do not fix the value $y$, but consider $y=y(a)\to \infty$ as $a\to 0$ sufficiently slow.
Olvera-Cravioto, Blanchet and Glynn \cite{OBG} showed that, if the increments possess regular varying tails with index $r>2$, there exists a critical value $y(a)\approx \sigma^2(r-2)a^{-1}\ln a^{-1}/2$, under which the heavy traffic approximation holds. 
Denisov and Kugler \cite{DK14} (see also \cite{BL14}) identified the critical value for general subexponential distributions, e.g. $y(a) \approx a^{-1/(1-\gamma)}$ in the Weibull case, where $\gamma\in (0,1)$ is the parameter of the Weibull distribution.

In this paper we assume that $X^{(a)}$ possesses a $\Delta$-lattice distribution, that means there exists some $\Delta>0$ such that $\P(X\in \Delta \mathbb Z)=1$ and $\Delta$ is the maximal positive number with this property. Let us assume without loss of generality that $\Delta$ is an integer. 
Our main result is a local limit theorem for the probability $\P(M^{(a)}=y\Delta)$ as $a\to 0$ for
$y$ such that $y\to \infty$ and $ay=O(1)$ under the assumption that the increments possess an aperiodic lattice distribution with zero-shift. 
The main idea for our proof is to find a a representation of the probability $\P(M^{(a)}=y\Delta)$ as a geometric sum and to derive and apply a uniform renewal theorem to find the asymptotic behaviour of this sum. This uniform renewal theorem will be a generalization of a result attained by Nagaev \cite{N68}. 

It is worth mentioning that the approach used in this paper is similar to the method used in \cite{BL14}, 
where the authors use
the well-known representation of $\P(M^{(a)}>y)$ as a geometric sum of independent random variables (see for example \cite{A00}) and a uniform renewal theorem from \cite{BG07} to establish the asymptotic behaviour of $\P(M^{(a)}>y)$ as $a\to 0$ and $y\to \infty$ for subexponential distributions.
In \cite{BG07} there is also a uniform renewal theorem used to develop asymptotic expansions of the distribution of a geometric sum.

We now state our main result.
\begin{theorem}
\label{T1}
 Assume that \eqref{assumption_weak_conv} and \eqref{assumption_moment_ex} hold and 
 suppose that $X^{(a)}$ possesses an aperiodic $\Delta$-lattice distribution for $a$ small enough.
 Then, as $a\to 0$,
 \begin{align}
  \P(M^{(a)}=y\Delta) 
  &\sim \frac{2a\Delta}{\sigma^2}\exp\left\{-\frac{2a y\Delta}{\sigma^2}\right\} \label{T1.2}
 \end{align}
 uniformly for all $y$ such that $y\to \infty$ and $ay=O(1)$ as $a\to 0$. 
\end{theorem}

In the non-local case, 
it is known (see for example Wachtel and Shneer \cite{WS11}) that one only needs to assume
$\lim_{a\to 0} \mathbf{Var} X^{(a)}=\sigma^2 \in (0,\infty)$ and 
a Lindeberg-type condition
$$
 \lim_{a\to 0} \E[(X^{(a)})^2;|X_1^{(a)}|>K/a] =0 \quad \text{for all } K>0
$$
to establish \eqref{heavy_traffic}.
This means that we must make stronger assumptions to establish our local result than it is needed in the non-local case.

Obviously,  Theorem \ref{T1} restates the heavy traffic asmyptotics \eqref{heavy_traffic}: 
As $a\to \infty$,
\begin{align*}
 \P(M^{(a)}\ge y\Delta) 
 &= \sum_{x=y}^\infty  \P(M^{(a)} = x\Delta)
 \sim \frac{2a\Delta}{\sigma^2} \sum_{x=y}^\infty e^{-2ax\Delta/\sigma^2}\\
 &= \frac{2a\Delta}{\sigma^2} \frac{e^{-2ay\Delta/\sigma^2}}{1-e^{-2ay\Delta/\sigma^2}}
 \sim e^{-2ay\Delta/\sigma^2}
\end{align*}
for all $y$ such that $y\to \infty$ and $ay = O(1)$ as $a\to 0$.


\section{Uniform renewal theorem}

In this section we prove a modification of Theorem 1 in Nagaev \cite{N68} which is, unlike the uniform renewal theorem from Nagaev, even uniform in the expected value. This renewal theorem is the key to the proof of our main result. 

Consider a family of non-negative $\Delta$-latticed and aperiodic random variables $\{Z^{(b)}, b \in I\}$ 
with $\E[Z^{(b)}]=b> 0$ and a non-empty set $I\subseteq \mathbb{R}$ that contains at least one accumulation point.
Denote by $F^{(b)}$ the distribution function of $Z^{(b)}$ and by $F_k^{(b)}$ the k-fold convolution of 
$F^{(b)}$ with itself. Let 
$$
 H(x,b,A) = \sum_{k=0}^\infty A^k F_k^{(b)}(x), \quad A>0.
$$
In renewal theory one usually studies the asymptotic behavior of $H(x+h,b,1) - H(x,b,1)$, $h>0$. 
However, the case $A\neq 1$ is also of great interest. 
Nagaev's motivation for studying the case $A \neq 1$ comes from branching processes, 
since there arises a need for an asymptotic representation for $H(x+h,b,A) - H(x,b,A)$ as $x\to \infty$
with an estimate for the remainder term which is uniform in $A$. 
For our purposes we seek a representation for $H(x+h,b,A) - H(x,b,A)$ as $x\to \infty$ and the estimate for the remainder shall be uniform in $A$ and $b$.
Assume that there exists some $s>1$ such that 
\begin{equation}
\label{moment_ex_Z}
 \sup_{b\in I} \E[(Z^{(b)})^s]<\infty.
\end{equation}
Put
$$
 f_{k\Delta}^{(b)}=F^{(b)}(k\Delta)-F^{(b)}((k-1)\Delta), \quad f_y^{(b)}(z)=\sum_{k=0}^y f_{k\Delta}^{(b)} z^{k}, 
$$
$$
 \mu_y^{(b)}(z) = f_y^{(b)'}(z) = \sum_{k=1}^y k f_{k\Delta}^{(b)} z^{k-1}.
$$

\begin{proposition}
\label{P1}
Let $\lambda_{y\Delta}^{(b)}(A)$ be the real non-negative root of the equation $Af_{y}^{(b)}(z)\\=1$. 
Assume that \eqref{moment_ex_Z} holds for some $s > 1$. 
Then, there exists a positive constant $\alpha$ for every accumulation point $b_0$ of $I$, such that
 \beq
 \label{P1.3}
   \sum_{k=1}^\infty A^k \left(F_k^{(b)}(y \Delta)-F_k^{(b)}((y-1)\Delta)\right)
  = \frac{\big(\lambda_{y\Delta}^{(b)}(A)\big)^{-y-1}}{A \mu_{y\Delta}^{(b)}(\lambda_{y\Delta}^{(b)}(A))} 
    + o(y^{-\min\{1,s-1\}}\ln y) 
 \eeq
 uniformly in $b\in I\cap \{b\in I: |b-b_0|\le \alpha\}$ and $A_y\le A\le 1$, where 
 \beq
 \label{P1.4}
   A_y = 1-C/y
 \eeq
 with a fixed positive number $C$.
\end{proposition}


\subsection{Proof of the uniform renewal theorem}

Although the uniform renewal theorem is a generalization of Theorem 1 in Nagaev \cite{N68}, the main idea of the proof is the same. However, for reasons of completeness, we give the whole proof.

Let us assume without loss of generality $\Delta=1$, $I=[0,b_1]$ with $b_1>0$ and that $y$ is sufficiently 
large in this section, even if it is not explicitely mentioned.
Throughout the following $\int_a^b g(x) dF^{(b)}(x)$ is to be interpreted as $\int_{a+}^{b+} g(x) dF^{(b)}(x)$.

\begin{lemma}
\label{L2} 
 Assume that \eqref{moment_ex_Z} holds for some $I$ and $s>1$. 
 Put $\mu^{(b)}=\E[Z^{(b)}]$, $b\in I$, and 
 $U_y(\delta)= \{z: 1\le |z| \le e^{h_y}, |\arg z|\le \delta \}$ for some $h_y=O(1/y)$.
 Then,
 \begin{equation}
 \label{L2_1}
  \lim_{\delta\to 0} \lim_{y\to \infty} \sup_{b\in I, z\in U_y(\delta)} |\mu_y^{(b)}(z)-\mu^{(b)}| = 0.
 \end{equation}
\end{lemma}

\begin{proof}
First of all,
\begin{align}
 |\mu_y^{(b)}(z)-\mu^{(b)}|
 &=\left|\int_0^y xz^{x-1} dF^{(b)}(x) - \int_{0}^\infty x dF^{(b)}(x)\right|\nonumber\\
 &\le \int_0^{y} x|z^{x-1}-1| dF^{(b)}(x)+\int_{y}^\infty x dF^{(b)}(x). \label{eq_0_6_0}
\end{align}
When $x,|z| \ge 1$, one can easily see by Taylor's approximation that  
$$
 |z^{x-1}-1| \le x|z-1||z|^x.
$$
Using this estimate we obtain for all $z\in U_y(\delta)$ and $N\le y$,
\begin{align*}
 \int_0^N x|z^{x-1}-1| dF^{(b)}(x) 
 &\le |z-1| \int_0^N x^2|z|^x dF^{(b)}(x)\\
 &\le |z-1|e^{h_y y} \int_0^N x^2 dF^{(b)}(x)
 \le N^2|z-1|e^{h_y y}.
\end{align*}
Further, a straightforward trigonometric calculation shows that for $\delta$ sufficiently small,
$$
 |z-1|
 \le |z-e^{i \arg z}| + |1-e^{i \arg z}|
 = |z|-1 + \sqrt{2(1-\cos(\arg z))}
 \le e^{h_y}-1 + 2\delta
$$
for all $z\in U_y(\delta)$ and hence, as $y\to \infty$,
$$
 \int_0^N x|z^{x-1}-1| dF^{(b)}(x)
 \le e^{h_y y} N^2 (e^{h_y}-1 + 2\delta)
 = e^{h_y y} N^2 (2\delta+h_y+o(h_y))
$$
uniformly in $b\in I$. At the same time, for $z\in U_y(\delta)$, 
assumption \eqref{moment_ex_Z} and $h_y y = O(1)$ imply that there exists an absolute number $K>0$ such that
for all $N\le y$,
\begin{align*}
 \int_N^{y} x|z^{x-1}-1| dF^{(b)}(x)
 &\le (1+e^{h_y y}) \int_N^{y} x dF^{(b)}(x)\\
 &\le \frac{1+e^{h_y y}}{N^{s-1}} \int_N^\infty x^{s} dF^{(b)}(x)
 \le K N^{1-s}
\end{align*}
and by setting $N=(2\delta+h_y)^{-1/3}$ and choosing $K_1$ such that $e^{h_y y}\le K_1$ (which is possible due to the assumption $h_y=O(1/y)$), 
we attain 
\begin{align}
 \int_0^{y} x|z^{x-1}-1| dF^{(b)}(x)
 &\le e^{h_y y}(2\delta+h_y)^{1/3} + K (2\delta+h_y)^{(s-1)/3} + o(h_y) \nonumber\\
 &\le 2^{1/3}K_1\delta^{1/3} + K 2^{(s-1)/3} \delta^{(s-1)/3} + o(1)\label{eq_0_7} 
\end{align}
uniformly in $b \in I$ as $y\to \infty$. 
Plugging the \eqref{eq_0_7} into \eqref{eq_0_6_0} and using \eqref{moment_ex_Z} once more, we conclude
\begin{align}
 |\mu_y^{(b)}(z)-\mu^{(b)}| 
 \le 2^{1/3}K_1 \delta^{1/3} + K 2^{(s-1)/3} \delta^{(s-1)/3} + o(1) \label{first_limit}
\end{align}
uniformly in $b\in I$ as $y\to \infty$. 
\end{proof}


\begin{lemma}
\label{L3}
Assume that \eqref{moment_ex_Z} holds for some $I$ and $s>1$. 
Then, for large enough $y$, $\lambda_y^{(b)}(A)<e^{h_y}$ for all $A_y \le A \le 1$ and $b\in I$, 
where $A_y = 1-C/y$ with some constant $C>0$ and $h_y=C_1/(\mu^{(0)} y)$ with $C_1 > C\mu^{(0)}/\inf_{b\in I}\mu^{(b)}$.
\end{lemma}

\begin{proof}
We want to estimate the difference $\lambda_y^{(b)}(A)-1$. 
First of all, by regarding the definition of $\lambda_y^{(b)}(A)$,
\begin{align*}
 &\int_{0-}^y \left((\lambda_y^{(b)}(A))^x-1\right) dF^{(b)}(x)
 = f_y^{(b)}(\lambda_y^{(b)}(A))- \int_{0-}^y dF^{(b)}(x)\\
 &\hspace{4cm}= \frac{1}{A} - 1 + \int_y^\infty dF^{(b)}(x)
 = \frac{1-A}{A} + \int_y^\infty dF^{(b)}(x).
\end{align*}
Further, $\lambda_y^{(b)}(A)\ge 1$ for $A\le 1$ and therefore by the binomial formula,
$$
 (\lambda_y^{(b)}(A))^x-1
 \ge x(\lambda_y^{(b)}(A)-1)\quad , x\ge 0. 
$$
Thus, uniformly in $b\in I$,
\begin{align}
 (\lambda_y^{(b)}(A)-1) \int_{0-}^y x dF^{(b)}(x)
 &\le \int_{0-}^y \left((\lambda_y^{(b)}(A))^x-1\right) dF^{(b)}(x)\nonumber\\
 &= \frac{1-A}{A} + \int_y^\infty dF^{(b)}(x)  
 = \frac{1-A}{A} + O(y^{-s}), \label{eq_0_9}
\end{align}
where we used \eqref{moment_ex_Z} in the laste line.
The condition $A_y\le A\le 1$ implies that 
$
 1-A \le C/y,
$ 
hence 
$$
 \frac{1}{A} = 1+ \frac{1-A}{A} = 1+O\left(\frac{1}{y}\right)
$$
and consequently
\begin{equation}
\label{eq_0_10}
 \frac{1-A}{A} \le \frac{C}{A y} = \frac{C}{y} + O\left(\frac{1}{y^2}\right).
\end{equation}
From the inequalities \eqref{eq_0_9}, \eqref{eq_0_10} and \eqref{moment_ex_Z} we conclude that
\begin{align*}
 \lambda_y^{(b)}(A)-1 
 &\le \frac{C/y+O(y^{-2})+O(y^{-s})}{\mu^{(b)}-\int_y^\infty x dF^{(b)}(x)}
 = \frac{C/(\mu^{(b)}y)}{1-O(y^{1-s})} + O({y^{-2}})+O(y^{-s})\\
 &= \frac{C}{\mu^{(b)}y} + O(y^{-2})+O(y^{-s})
 < \frac{C_1}{\mu^{(0)} y}
\end{align*}
uniformly in $b\in I$ for all $y$ large enough. 
Therefore, since $e^x-1 \ge x$ for all $x>0$, $\lambda_y^{(b)}(A) < e^{h_y}$ 
uniformly in $A_y\le A \le 1$ and $b\in I$, if $y$ is sufficiently large.
\end{proof}


\begin{lemma}
\label{L4}
 Assume that \eqref{moment_ex_Z} holds for some $I$ and $s>1$. 
 Put $h_y=C_1/(\mu^{(0)} y)$ with a constant $C_1>C\mu^{(0)}/\inf_{b\in I} \mu^{(b)}$. 
 Then, there exists some $b_2>0$ such that for $y$ large enough, $Af_y^{(b)}(z)-1$ 
has no other zeros in the disc $|z|< e^{h_y}$  apart from $\lambda_y^{(b)}(A)$ and this holds
 uniform in $A_y\le A\le 1$ and $0\le b\le b_2$.
\end{lemma}

\begin{proof}
First of all, for all $|z|\le e^{h_y}$,
$$
 |\mu_y^{(b)}(z)| 
 \le \int_0^y x |z|^{x-1} dF^{(b)}(x)
 \le e^{h_y y} \mu^{(b)}. 
$$
Using in addition $h_y y=O(1)$ and \eqref{moment_ex_Z}, we conclude
\begin{equation}
\label{mu_a_finite}
 \sup_{y, b \le b_1, |z|\le e^{h_y}} |\mu_y^{(b)}(z)| <\infty.
\end{equation}
Therefore,
\begin{align}
 &\lim_{y\to \infty} \sup_{b\le b_1} \sup_{\stackrel{1\le r \le e^{h_y}}{0\le \varphi\le 2\pi}} 
  \left|f_y^{(b)}(re^{i\varphi})-f_y^{(b)}(e^{i\varphi})\right| \nonumber\\
 &=  \lim_{y\to \infty} \sup_{b\le b_1} \sup_{\stackrel{1\le r \le e^{h_y}}{0\le \varphi\le 2\pi}}
  \big|\mu_y^{(b)}(e^{i\varphi})\big|\left|re^{i\varphi}-e^{i\varphi}\right|  
 =0. \label{eq_0_12}
\end{align}
On the other hand, 
\begin{align}
 \lim_{y\to \infty} \sup_{b\le b_1} \sup_{0\le \varphi\le 2\pi} \big|f_y^{(b)}(e^{i\varphi})-f_\infty^{(b)}(e^{i\varphi})\big|
 &\le  \lim_{y\to \infty} \sup_{b\le b_1} \sup_{0\le \varphi\le 2\pi} \int_y^\infty \big|e^{i\varphi x}\big| dF^{(b)}(x)\nonumber\\ 
 &=\lim_{y\to \infty} \sup_{b\le b_1} \olF^{(b)}(y) = 0. \label{eq_0_13}
\end{align}
As $b\to 0$, $F^{(b)}(\cdot) \to F^{(0)}(\cdot)$ in the sense of Definition 3 from chapter VIII.1 in Feller \cite{F71} 
and $F^{(0)}$ is not defective because of \eqref{moment_ex_Z}. 
Obviously, $u_\varphi(\cdot)=e^{i\varphi\cdot}$ is equicontinuous with $|u_\varphi| = 1<\infty$. 
Hence, by a corollary in chapter VIII.1 in Feller \cite{F71}, 
\beq
\label{eq_0_13_1}
 \int_0^\infty e^{i\varphi x} dF^{(b)}(x) \to  \int_0^\infty e^{i\varphi x} dF^{(0)}(x) 
\eeq
uniformly in $0\le \varphi\le \pi$ as $b\to 0$. 

Now, let us first consider values of $z$ in the circle $|z|<e^{h_y}$ that are not in the vicinity of  $\lambda_y^{(b)}(A)$. 
Due to Lemma \ref{L3}, these values can be characterized as those values that satisfy $|z|<e^{h_y}$ and $\delta\le |\text{arg} z|\le \pi$, $\delta>0$. It is
$$
 \sup_{\delta \le \varphi \le \pi}\big|f_\infty^{(0)}(e^{i\varphi})\big| 
 = \sup_{\delta \le \varphi \le \pi}\left|\int_0^\infty e^{i\varphi x} dF^{(0)}(x)\right| 
 < \sup_{\delta \le \varphi \le \pi}\int_0^\infty \big|e^{i\varphi x}\big| dF^{(0)}(x)
 = 1.
$$
Combining the latter inequality with \eqref{eq_0_13_1}, we conclude that there exists some $b_2>0$ (assume without loss of generality $b_2\le b_1$) such that
$$
 \sup_{b\le b_2} \sup_{\delta \le \varphi \le \pi} |f_\infty^{(b)}(e^{i\varphi})| < 1
$$
and since this inequality is strict,
\begin{equation}
\label{eq_0_14}
 m(\delta) 
 := \inf_{b\le b_2} \inf_{A\le 1} \inf_{\delta\le \varphi\le \pi} \big|Af_\infty^{(b)}(e^{i\varphi})-1\big| > 0. 
\end{equation}
By combining \eqref{eq_0_12}, \eqref{eq_0_13} and \eqref{eq_0_14},
we conclude that for large enough $y$ and $A\in \mathfrak{A}_y$, 
\begin{equation}
\label{eq_0_15}
 \inf_{b\le b_2} \inf_{\stackrel{1\le r \le e^{h_y}}{\delta\le \varphi\le 2\pi}} \big|Af_y^{(b)}(re^{i\varphi})-1\big| > \frac{m(\delta)}{2}>0.
\end{equation}
On the basis of \eqref{eq_0_15} we can assert that if $Af_y^{(b)}(z)-1$ has a zero $\tilde{\lambda}_y^{(b)}(A)$ in the disc $|z|\le e^{h_y}$ differing from $\lambda_y^{(b)}(A)$, then $\tilde{\lambda}_y^{(b)}(A)$ will lie outside the region $\{z:1\le |z|\le e^{h_y}, |\arg z| \ge \delta\}$
and this holds uniformly in $b\in [0,b_2]$ and $A \in \mathfrak{A}_y$.\\
Next, consider the region $U_y(\delta)=\{z:1\le |z|\le e^{h_y}, |\arg z| < \delta\}$. 
Observe that Taylor's formula implies
\begin{align*}
 Af_y^{(b)}(z)-1
&= Af_y^{(b)}(z)-Af_y^{(b)}(\lambda_y^{(b)}(A))
\ge A\mu_y^{(b)}(\lambda_y^{(b)}(A))(z-\lambda_y^{(b)}(A)).
\end{align*}
This inequality plus the equicontinuity of $f_y^{(b)}(z)$ imply the existence of a
$\delta_1(b, A)>0$ such that 
$
 |Af_y^{(b)}(z)-1|
$
has no other zeros in the disc $|z-\lambda_y^{(b)}(A)|\le \delta_1(b,A)$ apart from 
$\lambda_y^{(b)}(A)$. Therefore, 
$$
 \tilde{m}(\delta_2):=\inf_{b\le b_2}  \inf_{A \in \mathfrak{A}_y} 
 \inf_{z:|z-\lambda_y{(b)}(A)|\le \delta_2}  |Af_y^{(b)}(z)-1| > 0.
$$ 
where $\delta_2 = \inf_{b\le b_2} \inf_{A\in \mathfrak{A}_y} \delta_1(b,A)>0$.
Observe that ${\lambda}_y^{(b)}(A)\ge 1$ for $A\le 1$ and ${\lambda}_y^{(b)}(A) < e^{h_y}$ by Lemma \ref{L3}. 
Hence, for $\delta$ small enough, say $\delta\le \delta_3$, 
the region 
$$
 \bigcup_{b\le b_2} \bigcup_{A\in \mathfrak{A}_y}\{z:|z-\lambda_y^{(b)}(A)|\le\delta_1\}
$$ 
covers $U_y(\delta)$
and that means $\tilde{\lambda}_y^{(b)}(A)$ cannot lie in the region
$\{z:1-\varepsilon_0\le |z|\le e^{h_y}, |\arg z|<\delta_3\}$. 
Setting $\delta=\delta_3$ in \eqref{eq_0_15} we conclude that 
$\tilde{\lambda}_y^{(b)}(A)$ cannot lie in the annulus $1\le |z|\le e^{h_y}$.
Since $|\tilde{\lambda}_y^{(b)}(A)|\ge 1$ for all $A\le 1$, we finally obtain that $\tilde{\lambda}_y^{(b)}(A)$ does not lie in the disc $|z|\le e^{h_y}$, so $\lambda_y^{(b)}(A)$ is the only root of the equation $Af_y^{(b)}(A)=1$ in the disc $|z|\le e^{h_y}$ and this holds uniformly in $b\le b_2$ and $A_y \le A \le 1$.
\end{proof}


{\it{Proof of Proposition \ref{P1}.}}
Let $\gamma_{y}$ be a circle of radius $r_{y}=e^{h_{y}}$ with $h_{y}=C_1/(\mu^{(0)} y)$, $C_1>\mu^{(0)}+C\mu^{(0)}/\inf_{b\le b_1}\mu^{(b)}$ and $C$ from \eqref{P1.4}. 
Then, according to Lemma \ref{L3} and Lemma \ref{L4}, there exists some $b_2>0$ such that for all 
$0\le b \le b_2$ and $A\in \mathfrak{A}_y$, the function 
$1-Af_{y}^{(b)}(z)$ is zero in the disc $|z|\le e^{h_{y}}$, if and only if $z=\lambda_{y}^{(b)}(A)$.
Hence, the Residue theorem states that
\begin{equation}
\label{residue}
 \frac{1}{2\pi i} \int_{\gamma_y} \frac{z^{-y-1}}{1-Af_{y}^{(b)}(z)} dz
 = \text{Res}\left(\frac{z^{-y-1}}{1-Af_{y}^{(b)}(z)},\lambda_{y}^{(b)}(A)\right) + \text{Res}\left(\frac{z^{-y-1}}{1-Af_{y}^{(b)}(z)}, 0\right).
\end{equation}
for $0\le b \le b_2$ and $A\in\mathfrak{A}_y$.

In the following denote by $C_n(f(z))$, $n\ge 1$, the coefficient of $z^n$ in the Taylor series of the function $f(z)$. 
An easy calculation shows that
\begin{equation*}
 A^n (f_\infty^{(b)}(z))^n = A^n \sum_{j=1}^\infty \left(F_n^{(b)}(j)-F_n^{(b)}(j-1)\right)z^j
\end{equation*}
and consequently, by changing the order of summation, it is not hard to see that
$$ \sum_{k=1}^\infty A^k  \left(F_k^{(b)}(n)-F_k^{(b)}(n-1)\right)
 =C_n\left(\frac{1}{1-Af_\infty^{(b)}(z)}\right).
$$
On the other hand, when $n\le y$, 
$$
 C_n\left(\frac{1}{1-Af_\infty^{(b)}(z)}\right) 
 = C_n\left(\frac{1}{1-Af_{y}^{(b)}(z)}\right)
$$
and thus, for $n\le y$,
\begin{equation}
\label{coefficient}
 \sum_{k=1}^\infty A^k  \left(F_k^{(b)}(n)-F_k^{(b)}(n-1)\right)
 =C_n\left(\frac{1}{1-Af_{y}^{(b)}(z)}\right).
\end{equation}
Regarding \eqref{coefficient} with $n=y$, one can easily verify
$$
 \text{Res}\left(\frac{z^{-y-1}}{1-Af_{y}^{(b)}(z)}, 0\right)
 = \sum_{k=1}^\infty A^k \left(F_k^{(b)}(y)-F_k^{(b)}(y-1)\right).
$$
The pole of the function $z^{-y-1}/(1-Af_{y}^{(b)}(z))$ in $z=\lambda_{y}^{(b)}(A)$ is of order $1$. 
Therefore, it is not hard to see that
$$
 \text{Res}\left(\frac{z^{-y-1}}{1-Af_{y}^{(b)}(z)}, \lambda_{y}^{(b)}(A)\right)
 =-\frac{\lambda_{y}^{(b)}(A)^{-y-1}}{A\mu_{y}^{(b)}(\lambda_{y}^{(b)}(A))}
$$
and by combining the latter results we obtain
 \begin{align*}
  \nonumber
  \sum_{k=1}^\infty A^k \left(F_k^{(b)}(y)-F_k^{(b)}(y-1)\right) 
  &= \frac{\big(\lambda_{y}^{(b)}(A)\big)^{-y-1}}{A \mu_{y}^{(b)}(\lambda_{y}^{(b)}(A))} 
   + \frac{1}{2\pi i} \int_{\gamma_{y}} \frac{z^{-y-1}}{1-Af_{y}^{(b)}(z)} dz.
 \end{align*}
It remains to show that under the conditions of Proposition \ref{P1}, 
\begin{equation}
 \label{rest_negligible}
 \frac{1}{2\pi i} \int_{\gamma_y} \frac{z^{-y-1}}{1-Af_{y}^{(b)}(z)} dz 
 = o\left(y^{-\min\{1,s-1\}}\ln y\right)
\end{equation}
uniformly in $b\le b_2$ and $A_y\le A\le 1$. 
Let
\begin{align*}
 &\varphi_y^{(b)}(z) = A(f_{y}^{(b)}(z)- f_{y}^{(b)}(r_y))-A\mu_{y}^{(b)}(r_y)(z-r_y),\\
 &\psi_y^{(b)}(z) = 1- A f_{y}^{(b)}(r_y)-A\mu_{y}^{(b)}(r_y)(z-r_y).
\end{align*}
Then, the following identity holds:
\begin{equation}
 \label{eq_4}
 \frac{1}{1-Af_y^{(b)}(z)} - \frac{1}{\psi_y^{(b)}(z)} 
 =\frac{\psi_y^{(b)}(z) - 1 + Af_{y}^{(b)}(z)}{(1-Af_y^{(b)}(z))\psi_y^{(b)}(z)}
 =\frac{\varphi_y^{(b)}(z)}{(1-Af_y^{(b)}(z))\psi_y^{(b)}(z)}.
\end{equation}
Let $\varepsilon>0$, $\gamma_y(\varepsilon)= \gamma_y \cap U_y(\varepsilon)$ and let $\overline{\gamma}_y(\varepsilon)$ be the complement of $\gamma_y(\varepsilon)$ with respect to $\gamma_y$.
By \eqref{eq_4},
\begin{align*}
 \int_{\gamma_y} \frac{z^{-y-1}}{1-Af_y^{(b)}(z)} dz 
 = \int_{\gamma_y} \frac{z^{-y-1}}{\psi_y^{(b)}(z)} dz 
  &+\int_{\gamma_y(\varepsilon)} \frac{z^{-y-1}\varphi_y^{(b)}(z)}{(1-Af_y^{(b)}(z))\psi_y^{(b)}(z)} dz\\
  &+\int_{\overline{\gamma}_y(\varepsilon)} \frac{z^{-y-1}\varphi_y^{(b)}(z)}{(1-Af_y^{(b)}(z))\psi_y^{(b)}(z)} dz.
\end{align*}
Using \eqref{eq_4} once again, the last integral of the latter identity can be rewritten as 
$$
 \int_{\overline{\gamma}_y(\varepsilon)} \frac{z^{-y-1}\varphi_y^{(b)}(z)}{(1-Af_y^{(b)}(z))\psi_y^{(b)}(z)} dz
 = -\int_{\overline{\gamma}_y(\varepsilon)} \frac{z^{-y-1}}{\psi_y^{(b)}(z)} dz 
  +  \int_{\overline{\gamma}_y(\varepsilon)} \frac{z^{-y-1}}{1-Af_y^{(b)}(z)} dz.
$$
Hence,
\begin{equation}
 \int_{\gamma_y} \frac{z^{-y-1}}{1-Af_y^{(b)}(z)} dz
 = I_1^{(b)}(y) + \sum_{j=2}^4 I_j^{(b)}(y, \varepsilon), \label{I_1234}
\end{equation}
where
\begin{align*}
 &I_1^{(b)}(y) = \int_{\gamma_y} \frac{z^{-y-1}}{\psi_y^{(b)}(z)} dz, 
 \quad I_2^{(b)}(y, \varepsilon) = \int_{\gamma_y(\varepsilon)} \frac{z^{-y-1}\varphi_y^{(b)}(z)}{(1-Af_y^{(b)}(z))\psi_y^{(b)}(z)} dz,\\
 &I_3^{(b)}(y, \varepsilon) =  -\int_{\overline{\gamma}_y(\varepsilon)} \frac{z^{-y-1}}{\psi_y^{(b)}(z)} dz,
 \quad I_4^{(b)}(y, \varepsilon) =  \int_{\overline{\gamma}_y(\varepsilon)} \frac{z^{-y-1}}{1-Af_y^{(b)}(z)} dz.
\end{align*}
To calculate $I_1^{(b)}$ let us examine integrals of the form
\begin{equation}
 \label{zero_int} 
 \int_{|z|=c^2} \frac{z^{-n}}{d z+h}dz,
\end{equation}
where $n>0$, $d,h\in \mathbb{C}$ and $|h|<c^2 |d|$. 
For $|h|<c^2|d|$, the function $z^{-n}/(dz+h)$ has exactly two singularities in the disc $|z|\le c^2$,
one in $0$ and the other in $-h/d$. 
Consequently the Residue theorem states that
\begin{equation*}
 \int_{|z|=c^2} \frac{z^{-n}}{d z+h}dz
 = \text{Res}\left(\frac{z^{-n}}{d z+h},0\right) + Res\left(\frac{z^{-n}}{d z+h},-\frac{h}{d}\right).
\end{equation*}
The pole in $z=0$ has order $n$, hence
$$
 \text{Res}\left(\frac{z^{-n}}{d z+h},0\right)
 =(-1)^{n-1} d^{n-1} h^{-n}
$$
and the pole in $z=-h/d$ is of order 1, thus
$$
 \text{Res}\left(\frac{z^{-n}}{d z+h},-\frac{h}{d}\right) = (-1)^n d^{n-1} h^{-n}.
$$
Therefore, 
\begin{equation}
 \label{int_equals_zero}
 \int_{|z|=c^2} \frac{z^{-n}}{d z+h}dz
 = [(-1)^{n-1} + (-1)^n] d^{n-1} h^{-n} 
 = 0.
\end{equation}
By the equicontinuity of $\mu_y^{(b)}(\cdot)$, the result from \eqref{mu_a_finite}, Lemma \ref{L2} and Lemma \ref{L3}, as $y\to \infty$,
\begin{align}
 f_y^{(b)}(r_y)-f_y^{(b)}(\lambda_y^{(b)}(A))
 &= (r_y- \lambda_y^{(b)}(A))\mu^{(b)}(\lambda_y^{(b)}(A)) + o(r_y- \lambda_y^{(b)}(A)) \nonumber\\
 &= (r_y- \lambda_y^{(b)}(A))\mu^{(b)} + o(r_y- \lambda_y^{(b)}(A)) \label{absolute_continuous}
\end{align}
uniformly in $b\le b_2$ and $A\in \mathfrak{A}_y$.  
By virtue of Lemma \ref{L3} and the definition of $C_1$, 
$
 |\lambda_y^{(b)}(A)| \le e^{h_y-1/y}
$ 
and consequently
\begin{align*}
 r_y-\lambda_y^{(b)}(A) 
&\ge e^{h_y}(1-e^{-1/y})\\
&=(1+h_y+o(y^{-1}))(y^{-1}+o(y^{-1})) 
= y^{-1} + o(y^{-1})
\end{align*}
uniformly in $b\le b_2$ and $A\in \mathfrak{A}_y$.
By plugging this results into \eqref{absolute_continuous},
\begin{equation}
 \label{eq_6}
  1-A f_{y}^{(b)}(r_y) \le - \frac{A \mu^{(b)}}{y} +o\left(\frac{1}{y}\right) <0
\end{equation}
for $y$ large enough. Now put  
$
 h=1-Af_{y}^{(b)}(r_y)+A\mu_{y}^{(b)} (r_y)r_y
$
and 
$
 d=-A\mu_{y}^{(b)} (r_y). 
$
Then, since $A\mu_{y}^{(b)} (r_y)r_y\ge A\mu_{y}^{(b)} (1)\neq o(1)$, we obtain by virtue of \eqref{eq_6},
$$
 |h|\le A\mu_{y}^{(b)} (r_y)r_y = |d|r_y
$$
and consequently by \eqref{int_equals_zero},
\begin{equation}
 \label{I1_equals_zero}
 I_1^{(b)}(y) = \int_{\gamma_y} \frac{z^{-y-1}}{\psi_y^{(b)}(z)} dz = 0.
\end{equation}
Let us now consider $I_2^{(b)}$.
Clearly,
$$
 I_2^{(b)}(y, \varepsilon) = i r_y^{-y}\int_{\varepsilon}^{2\pi-\varepsilon} 
 \frac{\varphi_y^{(b)}(r_y e^{it})}{(1-Af_y^{(b)}(r_y e^{it}))\psi_y^{(b)}(r_y e^{it})} e^{-ity} dt.
$$
Taiblesons \cite{T67} estimate for Fourrier coefficients states that for any function $f$ 
with bounded variation on $[0, 2\pi]$ and
$
 f(x) \sim \sum_{n=-\infty}^{\infty} c_n e^{inx}
$
as $x\to \infty$, it is 
$$
 |c_n| \le \frac{2\pi \text{var}(f)}{n},
$$
where var denotes the variation of $f$,  defining this to be the sum of the variations of the real and the imaginary parts.
Hence, 
\begin{equation}
\label{varbound}
 I_2^{(b)}(y,\varepsilon)
 = O\left(\frac{1}{y} \underset{z \in \gamma_y(\varepsilon)}{\var} \frac{\varphi_y^{(b)}(z)}{(1-Af_y^{(b)}(z))\psi_y^{(b)}(z)}\right).  
\end{equation}
The variation of $\omega_y^{(b)}(z):=\varphi_y^{(b)}(z)/((1-Af_y^{(b)}(z))\psi_y(z))$ on $\gamma_y(\varepsilon)$ can be rewritten as follows:
\begin{align*}
 \underset{z \in \gamma_y(\varepsilon)}{\var}  \omega_y^{(b)}(z)
 &=  \underset{z \in \gamma_y(\varepsilon)}{\var} \Re (\omega_y^{(b)}(z))
  +  \underset{z \in \gamma_y(\varepsilon)}{\var} \Im (\omega_y^{(b)}(z))\\
 &= \int_{\gamma_y(\varepsilon)} \left(\left|\frac{d}{dl} \Re(\omega_y^{(b)}(z))\right| 
  + \left|\frac{d}{dl} \Im(\omega_y^{(b)}(z))\right|\right)dl,
\end{align*}
where $dl$ is the differential of the arc along $\gamma_y(\varepsilon)$. 
Due to the binomial formula, 
\begin{align*}
 \left(\left|\frac{d}{dl} \Re(\omega_y^{(b)}(z))\right| 
  + \left|\frac{d}{dl} \Im(\omega_y^{(b)}(z))\right|\right)^2
&\le 2\left(\left|\frac{d}{dl} \Re(\omega_y^{(b)}(z))\right|^2 
  + \left|\frac{d}{dl} \Im(\omega_y^{(b)}(z))\right|^2\right)\\
 &= 2 \left|\frac{d}{dz} \omega_y^{(b)}(z)\right|^2
\end{align*}
and thus, 
\begin{align}
 &\underset{z \in \gamma_y(\varepsilon)}{\var}  \frac{\varphi_y^{(b)}(z)}{(1-Af_y^{(b)}(z))\psi_y^{(b)}(z)}
 \le \sqrt{2} \int_{\gamma_y(\varepsilon)} \left| \frac{d}{dz}\frac{\varphi_y^{(b)}(z)}{(1-Af_y^{(b)}(z))
  \psi_y^{(b)}(z)}\right| dz\nonumber\\
 &\le \sqrt{2} \left(\int_{\gamma_y(\varepsilon)} \left| \frac{{\psi_y^{(b)}}'(z)\varphi_y^{(b)}(z)}
  {(1-Af_y^{(b)}(z))(\psi_y^{(b)}(z))^2}\right| dz
  +\int_{\gamma_y(\varepsilon)}\left| \frac{A\mu_y^{(b)}(z)\varphi_y^{(b)}(z)}
   {(1-Af_y^{(b)}(z))^2\psi_y^{(b)}(z)}\right| dz\right.\nonumber\\
 &\hspace{6cm}+\left.\int_{\gamma_y(\varepsilon)}\left| \frac{{\varphi_y^{(b)}}'(z)}{(1-Af_y^{(b)}(z)) \psi_y^{(b)}(z)}\right| dz\right)\nonumber\\
 &= \sqrt{2} (I_{21}^{(b)}+I_{22}^{(b)}+I_{23}^{(b)}).  \label{I2_123}
\end{align}
Let us bound the terms appearing in the integrands of the integrals from the latter inequality.
Using the definition of the complex absolute value, an easy calculation shows that
\begin{align*}
 |Af_y^{(b)}(z)-1|^2 
 &= A^2|f_y^{(b)}(z)-f_y^{(b)}(r_y)|^2 + |Af_y^{(b)}(r_y)-1|^2 \\
  &\hspace{2.5cm}- 2A(Af_y^{(b)}(r_y)-1)\Re(f_y^{(b)}(r_y)-f_y^{(b)}(z)).
\end{align*}
By the equicontinuity and Lemma \ref{L2}, as $y\to \infty$,
\begin{equation}
\label{absolute_cont_1}
 |f_y^{(b)}(r_y)-f_y^{(b)}(z)| = |r_y-z| \mu^{(b)}(z) + o(r_y-z) \ge (1-\delta)\mu^{(b)} |z-r_y|
\end{equation}
and 
\begin{equation}
\label{absolute_cont_2}
 |f_y^{(b)}(r_y)-f_y^{(b)}(z)| = |r_y-z| \mu^{(b)}(z) + o(r_y-z) \le (1+\delta)\mu^{(b)} |z-r_y|
\end{equation}
uniformly in $b\le b_2$ and $z\in U_y^{(b)}(\delta)$, if $\delta$ is small enough. 
Further, for all $z\in U_y(\delta)$ with $\delta$ sufficiently small,
$$
 \Re(r_y-z)= \sin(\arg z)|z-r_y| \le \delta |z-r_y|.
$$
By the virtue of \eqref{eq_6}, \eqref{absolute_cont_1} and \eqref{absolute_cont_2}, 
\begin{align*}
 |Af_y^{(b)}(z)-1|^2 
 \ge |1-Af_y^{(b)}(r_y)|^2 &+ (1-\delta)(\mu^{(b)})^2 A^2 |z-r_y|^2\\
  &- 2\delta (1+\delta) \mu^{(b)} A (Af_y^{(b)}(r_y)-1)|z-r_y|
\end{align*}
and by the binomial formula for $\delta$ sufficiently small,
$$
 2\mu^{(b)} A (Af_y^{(b)}(r_y)-1)|z-r_y|
 \le  (\mu^{(b)})^2 A^2 |z-r_y|^2 + |1-Af_y^{(b)}(r_y)|^2.
$$
Therefore, again by the binomial formula,
\begin{align*}
 |Af_y^{(b)}(z)-1|^2 
 &\ge(1-\delta-\delta(1+\delta))\left[|1-Af_y^{(b)}(r_y)|^2+(\mu^{(b)})^2 A^2 |z-r_y|^2\right]\\
 &\ge \frac{1-\delta-\delta(1+\delta)}{2}\left[|1-Af_y^{(b)}(r_y)|+A\mu^{(b)} |z-r_y|\right]^2. 
\end{align*}
Since $\delta$ can be chosen arbitrary small, one can especially choose $\delta$ so small that $1-\delta-\delta(1+\delta) \ge 1/2$. 
Thus,
\begin{equation}
 \label{eq1_1}
 |Af_y^{(b)}(z)-1| \ge \frac{|Af_y^{(b)}(r_y)-1|}{2} +\frac{A\mu^{(b)}  |z-r_y|}{2}
\end{equation}
uniformly in $b\le b_2$ and $z\in U_y^{(b)}(\delta)$. We proceed analogously to bound $|\psi_y^{(b)}(z)|$ for $z\in U_y(\delta)$ from below. 
It is $\Re(r_y-)z\le |z-r_y|$ and by virtue of Lemma \ref{L2}, $\mu_{y}^{(b)}(r_y) \in [(1-\widehat{\delta}_1)\mu^{(b)},(1+\widehat{\delta}_1)\mu^{(b)}]$ for $y$ large enough. Consequently, one can easily see that for $\widehat{\delta}_1$ small enough,
\begin{align}
 |\psi_y^{(b)}(z)|^2 
 &= |1-f_{y}^{(b)}(r_y)|^2+A^2\left(\mu_{y}^{(b)}(r_y)\right)^2 |z-r_y|^2\nonumber\\ 
 &\hspace{3cm}- 2A(f_{y}^{(b)}(r_y)-1) \mu_{y}^{(b)}(r_y)\Re(r_y-z)\nonumber\\
 &\ge \frac{1-\widehat{\delta}_2}{2}\left[|1-f_{y}^{(b)}(r_y)|+A\mu^{(b)} |z-r_y|\right]^2\label{eq1_1_2}
\end{align}
for all $\widehat{\delta}_2\le 1/2$.
Hence, 
\begin{equation}
 \label{eq1_2}
 |\psi_y^{(b)}(z)| \ge \frac{|1-Af_{y}^{(b)}(r_y)|}{4} +\frac{A \mu^{(b)} |z-r_y|}{4}.
\end{equation}
On the other hand, 
one can easily see that for every $z$ on $\gamma_y(\varepsilon)$ 
with $\varepsilon$ sufficiently small,  
\begin{align}
 |z-r_y| \ge |e^{i \arg z}-1| 
 &= \sqrt{\sin^2(\arg z) + (1-\cos(\arg z))^2} \nonumber\\
 &=\sqrt{2-2\cos (\arg z)} \ge \frac{|\arg z|}{2}, \label{abs}
\end{align}
where we used $\cos \varphi \le 1- \varphi^2/8$ in the last inequality. 
Combining inequalities \eqref{eq_6} and \eqref{abs} with \eqref{eq1_1}, we obtain
\begin{equation}
 \label{f_y_z}
 |1-Af_y^{(b)}(z)| \ge \frac{A\mu^{(b)}}{4}\left(\frac{1}{y}+|\arg z|\right).
\end{equation}
for $b\le b_2$ and $z\in U_y^{(b)}(\delta)$. 
The inequalities \eqref{eq_6}, \eqref{abs} and \eqref{eq1_2} provide
\begin{equation}
 \label{psi_y_z}
 |\psi_y^{(b)}(z)| \ge \frac{A\mu^{(b)}}{8}\left(\frac{1}{y}+|\arg z|\right)
\end{equation}
and, moreover, an easy calculation shows
\begin{equation}
 \label{psi_dash}
 |{\psi_y^{(b)}}'(z)|=A \mu_{y}^{(b)}(r_y)\le e^ {h_y y}A\mu^{(b)}.
\end{equation}
For $z\in U_y(\delta)$, 
\begin{equation*}
 |{f_{y1}^{(b)}}''(z)| 
 \le \left\{
     \begin{array}{ll}
       e^{h_y y} \E (Z^{(b)})^2 & : s \ge 2\\
       e^{h_y y}  y^{2-s} \E (Z^{(b)})^{s} & : 1<s<2
     \end{array}
   \right. 
\end{equation*}
and consequently, 
$$
 \frac{\varphi_y^{(b)}(z)}{|z-r_y|^2 y^{\max\{0,2-s\}}} 
 \sim \frac{{\varphi_y^{(b)}}'(z)}{2|z-r_y| y^{\max\{0,2-s\}}}
 \sim \frac{A{f_{y}^{(b)}}''(z)}{2y^{\max\{0,2-s\}}}=O(1)
$$
as $y\to \infty$. By virtue of \eqref{abs},
$$
 |e^{i\arg z}| = \sqrt{2-2\cos(\arg z)} \le \arg z,
$$
if $\arg z$ is sufficiently small.
Hence, if $\arg z$ is is sufficiently small,
\begin{equation}
 \label{phi}
 \varphi_y^{(b)}(z) = O(y^{\max\{0,2-s\}} |z-r_y|^2) = O(y^{\max\{0,2-s\}}\arg^2(z))
\end{equation}
and
\begin{equation}
\label{phi_dash}
 {\varphi_y^{(b)}}'(z) = O(y^{\max\{0,2-s\}} |z-r_y|) = O(y^{\max\{0,2-s\}}|\arg(z)|)
\end{equation}
uniformly in $b\le b_2$ and $A\in \mathfrak{A}_y$.
Considering \eqref{f_y_z}, \eqref{psi_y_z}, \eqref{psi_dash}, \eqref{phi} and $h_y y=O(1)$ provides
\begin{align*}
 |I_{21}^{(b)}|
 &\le r_y\int_{-\varepsilon}^\varepsilon \frac{\big|{\psi_y^{(b)}}'(r_y e^{it})\big||\varphi_y^{(b)}(r_y e^{it})|}
  {|f_y^{(b)}(r_y e^{it})-1||\psi_y^{(b)}(r_y e^{it})|^2} dt
= O\left(y^{\max\{0,2-s\}}\int_0^\varepsilon \frac{t^2}{(y^{-1}+t)^3} dt \right)
\end{align*}
uniformly in $b\le b_2$ and $A\in \mathfrak{A}_y$.  
Moreover,
\begin{align}
 \label{eq1_4}
 \int_0^\varepsilon \frac{t^2}{(y^{-1}+t)^3} dt
 &=\int_{1/y}^{\varepsilon+1/y} \frac{(w-y^{-1})^2}{w^3} dw\nonumber\\ 
 &\sim \ln(\varepsilon+y^{-1})-\ln(y^{-1}) 
 =\ln(1+\varepsilon y)
 \sim \ln(y)
\end{align}
and therefore, uniformly in $b\le b_2$ and $A\in \mathfrak{A}_y$,
\begin{equation}
 \label{I_21}
 |I_{21}^{(b)}| = O(y^{\max\{0,2-s\}}\ln y).
\end{equation}
In analogy, by additionally taking into account that $\mu_y^{(b)}(z)\le 2\mu^{(b)}$ 
due to Lemma \ref{L2} for $y$ large enough, one can easily see that
\begin{equation}
 \label{I_22}
 I_{22}^{(b)}
 = O(y^{\max\{0,2-s\}}\ln y)
\end{equation}
and furthermore, by regarding \eqref{phi_dash},
\begin{equation}
 \label{I_23}
 I_{23}^{(b)}
 = O\left(y^{\max\{0,2-s\}}\int_0^\varepsilon \frac{t}{(y^{-1}+t)^2} dt\right)
 = O(y^{\max\{0,2-s\}}\ln y).
\end{equation}
Finally, plugging \eqref{I_21}, \eqref{I_22} and \eqref{I_23} into \eqref{I2_123} we attain
$$
  \underset{z \in \gamma_y(\varepsilon)}{\var} \frac{\varphi_y^{(b)}(z)}{(1-Af_y^{(b)}(z))\psi_y^{(b)}(z)}
 = O(y^{\max\{0,2-s\}}\ln y)
$$
and hence by \eqref{varbound},
\begin{equation}
 \label{I2_negligible}
 |I_2^{(b)}(y,\varepsilon)| = o\big(y^{\max\{-1,-(s-1)\}}\ln y\big)
\end{equation}
uniformly in $b\le b_2$ and the admissible values of $A$. 
Next, we draw our attention to the integral $I_3^{(b)}$.
\begin{equation}
 I_3^{(b)}(y,\varepsilon)
 = -i r_y^{-y} \int_{\varepsilon\le |t| \le \pi} \frac{e^{-iyt}}{\psi_y^{(b)}(r_y e^{it})} dt. \label{eq1_6}
\end{equation}
To bound this integral we use Taibleson's estimate for Fourier coefficients again:
\begin{equation}
\label{eq1_7}
 \int_{\varepsilon\le |t| \le \pi} \frac{e^{-iyt}}{\psi_y^{(b)}(r_y e^{it})} dt
 =O\left(\frac{1}{y} \underset{z\in \overline{\gamma}_y(\varepsilon)}{\var} \frac{1}{\psi_y^{(b)}(z)} \right) 
\end{equation}
In analogy to \eqref{I2_123}, one can show that
\begin{align*}
 \underset{z\in \overline{\gamma}_y(\varepsilon)}{\var} \frac{1}{\psi_y^{(b)}(z)} 
 \le \sqrt{2} \int_{\overline{\gamma}_y(\varepsilon)} \left|\frac{d}{dz}\frac{1}{\psi_y^{(b)}(z)}\right| dz
 \le \sqrt{2} \int_{\overline{\gamma}_y(\varepsilon)} \frac{|{\psi_y^{(b)}}'(z)|}{|\psi_y^{(b)}(z)|^2} dz.
\end{align*}
By \eqref{eq1_2},
$$
 |\psi_y^{(b)}(z)|^2
 \ge \frac{A^2(\mu^{(b)})^2}{16} |z-r_y|^2
 \ge \frac{A^2 (\mu^{(b)})^2}{16} \varepsilon^2,
$$
where we used that $|z-r_y|>\varepsilon$ for all $z\in \overline{\gamma}_y(\varepsilon)$.
Therefore, by \eqref{psi_dash},
$$
 \underset{z\in \overline{\gamma}_y(\varepsilon)}{\var} \frac{1}{\psi_y^{(b)}(z)} 
 = O(1)
$$
and consequently by combining this result with \eqref{eq1_6}, \eqref{eq1_7} and $h_y y =O(1)$,
\begin{equation}
\label{I3_negligible}
 |I_3^{(b)}(y,\varepsilon)| 
 = O\left(\frac{1}{y}\right)
\end{equation}
uniform in $b\le b_2$ and $A\in \mathfrak{A}_y$.
It remains to consider $I_4^{(b)}$.
\begin{equation}
 |I_4^{(b)}(y,\varepsilon)|
 = \left|i r_y^{-y} \int_{\varepsilon\le |t| \le \pi} \frac{e^{-iyt}}{1-Af_y^{(b)}(r_y e^{it})} dt\right|
 = O\left(\frac{1}{y} \underset{z\in \overline{\gamma}_y(\varepsilon)}{\var} \frac{1}{1-Af_y^{(b)}(z)} \right). \label{eq1_8}
\end{equation}
Further, by \eqref{mu_a_finite} and \eqref{eq_0_15}, 
\begin{align*}
 \underset{z\in \overline{\gamma}_y(\varepsilon)}{\var} \frac{1}{1-Af_y^{(b)}(z)} 
 &\le \sqrt{2} \int_{\overline{\gamma}_y(\varepsilon)} \left|\frac{d}{dz}\frac{1}{1-Af_y^{(b)}(z)}\right| dz\\
 &= \sqrt{2} \int_{\overline{\gamma}_y(\varepsilon)} \frac{|\mu_y^{(b)}(z)|}{|1-Af_y^{(b)}(z)|^2} dz
 =O(1)
\end{align*}
and consequently
\begin{equation}
\label{I4_negligible}
 I_4^{(b)}(y,\varepsilon) 
 =O\left(\frac{1}{y}\right).
\end{equation}
Finally, by plugging the results attained in \eqref{I1_equals_zero}, \eqref{I2_negligible}, \eqref{I3_negligible} and \eqref{I4_negligible} into \eqref{I_1234}, we get
\begin{align}
 \int_{\gamma_y} \frac{z^{-y-1}}{1-Af_y^{(b)}(z)} dz
 &= o(y^{\max\{-1,-(s-1)\}}\ln y) + o(y^{-(s-1)}) + O(y^{-1})\nonumber\\ 
 &= o(y^{-\min\{1,s-1\}}\ln y) \label{eq1_9}
\end{align}
uniformly in $0\le b\le b_2$ and $A_y\le A \le 1$.


\section{Proof of the local limit theorem}

Put $\tau_{+,0}^{(a)}=0$ and define recursively for $i\ge 1$ the $i$-th strict ascending ladder epoch of the random walk $S^{(a)}$ and its corresponding ladder height by 
$$
 \tau_{+,i}^{(a)}:=\min\{k\ge \tau_{+,i-1}^{(a)}: S^{(a)}_k > S^{(a)}_{\tau_{+,i-1}}\} \quad \text{and}
  \quad \chi^{(a)}_i=S_{\tau_{+,i}^{(a)}}^{(a)}-S_{\tau_{+,i-1}^{(a)}}^{(a)}.
$$ 
In the case $i=1$ we write $\tau_+^{(a)}$ and $\chi^{(a)}$ instead of $\tau_{+,1}^{(a)}$ and $\chi^{(a)}_1$ respectively and, if additionally $a=0$,
we write $\tau_+$ and $\chi$ instead of $\tau_+^{(0)}$ and $\chi^{(0)}$ respectively.
Define random variables $Z_i^{(a)}$ as $iid$ copies of a random variable $Z^{(a)}$ with
$$
 \P(Z^{(a)}\in \cdot) = \P(\chi^{(a)}_1 \in \cdot | \tau_+^{(a)}<\infty).
$$

Denote by $\theta:= \min\{k\ge 0:S_k^{(a)}=M^{(a)}\}$ the first time the random walk reaches its maximum. 
Then,
$$
 \P(M^{(a)}=y\Delta) = \sum_{n=1}^\infty \P(M^{(a)}=y\Delta,\theta=n).
$$
We further define $M_n^{(a)}:=\max_{k\le n} S^{(a)}_k$ and $\theta_n:=\min\{k\le n:S_k^{(a)}=M^{(a)}_n\}$. 
By the Markov property,
$$
 \P(M^{(a)}=y\Delta,\theta=n)
 = \P(S_n^{(a)}=y\Delta,\theta_n=n)\P(\tau_a^+=\infty).
$$
Hence the following representation holds for the maximum:
\beq
\label{repr_for_max}
 \P(M^{(a)}=y\Delta) 
 = \P(\tau_+^{(a)}=\infty)\sum_{n=1}^\infty \P(S_n^{(a)}=y\Delta,\theta_n=n).
\eeq
Clearly,
 \begin{align}
  &\P(S_n^{(a)}=y\Delta, \theta_n=n) 
 = \P(S_n^{(a)}=y\Delta, n \text{ is a strict ascending ladder epoch})\nonumber\\
  &\hspace{1cm}=\sum_{k=1}^\infty \P(\chi^{(a)}_1+\chi^{(a)}_2+\dots + \chi^{(a)}_k=y\Delta, \tau_{+,1}^{(a)} + \tau_{+,2}^{(a)}+ \dots +\tau_{+,k}^{(a)}=n). \label{eq4.4}
 \end{align}
Denote the distribution function of $Z^{(a)}$ by $F^{(\mu^{(a)})}$, where $\mu^{(a)}=\E[Z^{(a)}]$,
and let $F^{(\mu^{(a)})}_k$ be the k-fold convolution of $F^{(\mu^{(a)})}$ with itself. 
Then, by using \eqref{eq4.4}, changing the order of summation and using the Markov property,
 \begin{align}
  &\sum_{n=1}^\infty \P(S_n^{(a)}=y\Delta, \theta_n=n)\nonumber\\ 
  &=\sum_{n=1}^\infty \sum_{k=1}^\infty \P(\chi^{(a)}_1+\chi^{(a)}_2+\dots + \chi^{(a)}_k=y\Delta, \tau_{+,1}^{(a)} + \tau_{+,2}^{(a)}+ \dots +\tau_{+,k}^{(a)}=n)\nonumber\\
  &=\sum_{k=1}^\infty \P(\chi^{(a)}_1+\chi^{(a)}_2+\dots + \chi^{(a)}_k=y\Delta | \tau_{+,k}^{(a)}<\infty)\P(\tau_{+,k}^{(a)}<\infty)\nonumber\\
  &=\sum_{k=1}^\infty A^k \left(F_k^{(\mu^{(a)})}(y\Delta)-F_k^{(\mu^{(a)})}((y-1)\Delta)\right) \label{eq4.5}
 \end{align}
with $A=\P(\tau_+^{(a)}<\infty)$. 
Combining results \eqref{repr_for_max} and \eqref{eq4.5} we attain
\begin{equation}
\label{geometric_sum}
 \P(M^{(a)}=y\Delta) 
 = \P(\tau_+^{(a)}=\infty)\sum_{k=1}^\infty A^k \left(F_k^{(\mu^{(a)})}(y\Delta)-F_k^{(\mu^{(a)})}((y-1)\Delta)\right).
\end{equation}
Next, we want to use Proposition \ref{P1} to determine the asymptotic behaviour of the sum on the right hand side of the latter equality.
Therefore, let us first show that under the assumptions of Theorem \ref{T1},
\begin{equation}
\label{weak_conv_Z}
 Z^{(a)} \xrightarrow[]{w} Z^{(0)}
\end{equation}
 as $a\to 0$. It is known that
\begin{equation}
 \P(\tau_+^{(a)}<\infty) \sim \P(\tau_+<\infty) = 1. \label{tau_a_to_tau}
\end{equation}
Thus, as $a\to 0$,
\begin{align*}
 \P(Z^{(a)}>x) 
 &=\frac{\P(\chi^{(a)}>x, \tau_+^{(a)}<\infty)}{\P(\tau_+^{(a)}<\infty)}
 \sim \P(\chi^{(a)}>x, \tau_+^{(a)}<\infty) 
\end{align*}
and, on the other hand, \eqref{assumption_weak_conv} and \eqref{tau_a_to_tau} imply that for every $R>0$, as $a\to 0$,
\begin{align*}
 \P(\chi^{(a)}>x, R<\tau_+^{(a)}< \infty) 
 &\le \P(R<\tau_+^{(a)}<\infty)\\
 &=\P(\tau_+^{(a)}<\infty)-\P(\tau_+^{(a)}\le  R)
 \sim \P(\tau_+> R).
\end{align*}
Further, by using \eqref{assumption_weak_conv} and the continuous mapping theorem,
\begin{align*}
 &\P(\chi^{(a)}>x, \tau_+^{(a)}\le R) 
 =\sum_{k=0}^{R-1} \P\big(S_{k+1}^{(a)}>x,\max_{1\le l\le k} S_l^{(a)}\le 0\big)\\
 &\hspace{2cm}\sim \sum_{k=0}^{R-1} \P\big(S_{k+1}>x,\max_{1\le l\le k} S_l\le0\big)
 = \P(\chi>x, \tau_+\le R)
\end{align*}
as $a\to 0$. Thus,
\begin{equation*}
 \limsup_{a\to 0} \P(Z^{(a)}>x) 
 \le \P(\chi>x, \tau_+\le R) + \P(\tau_+> R)
\end{equation*}
and by letting $R\to \infty$ we conclude 
\begin{equation*}
 \limsup_{a\to 0} \P(Z^{(a)}>x) 
 \le  \P(\chi>x, \tau_+< \infty) = \P(Z^{(0)} > x).
\end{equation*}
On the other side, the above calculations give
\begin{align*}
 \liminf_{a\to 0} \P(Z^{(a)}>x) 
 \ge \liminf_{a\to 0} \P(\chi^{(a)}>x,\tau_+^{(a)}\le R)
 = \P(\chi>x,\tau_+\le R)
\end{align*}
and by letting $R\to \infty$,
$$
 \liminf_{a\to 0} \P(Z^{(a)}>x) 
 \ge \P(\chi>x,\tau_+< \infty)
 = \P(Z^{(0)}>x).
$$
This means that \eqref{weak_conv_Z} holds under our assumptions.

Due to relation (16) of Chow \cite{C86} there exists a constant $C$ such that
\begin{align*}
 \E[\big(S^{(a)}_{\tau_+^{(a)}}\big)^{1+\varepsilon};\tau_+^{(a)}<\infty]
 \le C \int_0^\infty\frac{u^{2+\varepsilon}}{\E\big[|S^{(a)}_{\tau_-^{(a)}}|\wedge u\big]} d\P(\max\{0,X^{(a)}\}<  u)
\end{align*}
Obviously, 
$$
 \E\big[|S^{(a)}_{\tau_-^{(a)}}|\wedge u\big] 
 \ge \E\big[|S^{(a)}_{\tau_-^{(a)}}|\wedge \Delta\big]
 \ge \P(S_1^{(a)}<0) > 0
$$
for all $u\ge \Delta$ and therefore 
$$
 \E[\big(S^{(a)}_{\tau_+^{(a)}}\big)^{1+\varepsilon};\tau_+^{(a)}<\infty]
 \le \frac{C}{\P(S_1<0)} \int_0^\infty u^{2+\varepsilon} d\P(\max\{0,X^{(a)}\}<  u).
$$
Hence, by virtue of \eqref{assumption_moment_ex},
\begin{equation}
\label{moment_ex_Z2}
  \sup_{a\le a_0} \E[(Z^{(a)})^{1+\varepsilon}]<\infty.
\end{equation}
The convergence from \eqref{weak_conv_Z} combined with \eqref{moment_ex_Z2} implies 
\begin{equation}
\label{conv_moment_Z}
 \mu^{(a)} \to \mu^{(0)}
\end{equation}
as $a\to 0$ by dominated convergence.
It is known 
that for all $a>0$ the stopping time $\tau_+^{(a)}$ is infinite with positive probability and that 
\begin{equation}
 \P(\tau_+^{(a)}=\infty) = 1/\E[\tau_-^{(a)}], \label{identity_tau_+}
\end{equation}
where $\tau_-^{(a)}=\min\{k\ge 1: S^{(a)}_k \le 0\}$ is the first weak descending ladder epoch.
Totally analoguously to \eqref{moment_ex_Z2}, one can use (15) from Chow \cite{C86} to show that 
the existence of the second moment in assumption \eqref{assumption_moment_ex} implies 
$\sup_{a\le a_0}\E[S^{(a)}_{\tau_-^{(a)}}]<\infty$. 
Hence, one can use dominated convergence to show that 
$$
 \E[S^{(a)}_{\tau_-^{(a)}}] \to \E[S_{\tau_-^{(0)}}]
$$ 
as $a\to 0$. 
Thus, using \eqref{identity_tau_+}, the known identity
\begin{equation}
\label{identity_for_variance}
 \frac{\sigma^2}{2} = -\mu^{(0)}\E[S_{\tau_-^{(0)}}]
\end{equation}
and Wald's identity imply that 
\begin{equation}
\label{wald}
 \P(\tau_+^{(a)}=\infty)=\frac{1}{\E[\tau_-^{(a)}]}\sim \frac{a}{-\E[S_{\tau_-^{(0)}}]}
 \sim \frac{2a\mu^{(0)}}{\sigma^2}.
\end{equation}
The assumption $ay=O(1)$ implies the existence of a constant $C$ such that $y \le C/a$. 
Therefore, by \eqref{wald},
\begin{equation}
 \P(\tau_+^{(a)}<\infty) \ge 1- \frac{3 C\mu^{(0)}}{\sigma^2 y} \label{eq4.6}
\end{equation}
for $a$ small enough.
Summing up the results from \eqref{conv_moment_Z} and \eqref{eq4.6},
this means that we can apply Proposition \ref{P1} for $I=\{\mu^{(a)}:0\le a \le a_0\}$ with 
$a_0>0$ small enough,
$A_y=1-3C\mu^{(0)}/(\sigma^2 y)$, 
$A=\P(\tau_+^{(a)}<\infty)$ and $s=1+\varepsilon$. 
Hence, 
\begin{equation}
\label{eq_nagaev}
 \sum_{k=1}^\infty A^k \left(F_k^{(\mu^{(a)})}(y\Delta)-F_k^{(\mu^{(a)})}((y-1)\Delta)\right)
 = \frac{\big(\lambda_{y\Delta}^{(a)}(A)\big)^{-y-1}}{A \mu_{y}^{(a)}(\lambda_{y\Delta}^{(a)}(A))} 
    + o(y^{-\min\{1,\varepsilon\}}\ln y)
\end{equation}
and consequently, by combining equations \eqref{geometric_sum}, \eqref{eq_nagaev}
and the fact that $1-A=O(a)$, we attain
\begin{align}
 \P(M^{(a)}=y\Delta)
 &= (1-A)\frac{\big(\lambda_{y\Delta}^{(a)}(A)\big)^{-y-1}}{A \mu_{y\Delta}^{(a)}(\lambda_{y\Delta}^{(a)}(A))} 
    + o(ay^{-\min\{1,\varepsilon\}}\ln y). \label{eq4.7}
\end{align}
Let us now determine $\lambda_{y\Delta}^{(a)}(A)$ and $\mu_{y}^{(a)}(\lambda_{y\Delta}^{(a)}(A))$. 
Write $\lambda_{y\Delta}$ and $\mu_{y}(\lambda_{y\Delta})$ instead of $\lambda_{y\Delta}^{(a)}(A)$ and 
$\mu_{y}^{(a)}(\lambda_{y\Delta}^{(a)}(A))$ respectively for abbreviation
and put $\lambda_{y\Delta}=e^{\theta_{y\Delta}}$. 
According to the definition of $\lambda_{y\Delta}$, we want to find $\theta_{y\Delta}$ such that 
\beq
\label{eq_1_0}
 \E[\exp\{\theta_{y\Delta} Z^{(a)}/\Delta\};Z^{(a)}\le y\Delta] = \frac{1}{A}. 
\eeq
It turns out we don't need an exact solution for this equation and it is sufficient to 
determine $\theta_y$ such that
\beq
\label{eq_1}
  \E[\exp\{\theta_{y\Delta} Z^{(a)}/\Delta\};Z^{(a)}\le y\Delta] = \frac{1}{A} + O(y^{-1-\varepsilon}). 
\eeq
By Taylor's formula,
\begin{align*}
 &\E[\exp\{\theta_{y\Delta} Z^{(a)}/\Delta\};Z^{(a)}\le y\Delta]\\
 &= 1+\frac{\theta_{y\Delta} \mu^{(a)}}{\Delta} -\P(Z^{(a)}>y\Delta)
  -\frac{\theta_{y\Delta}}{\Delta} \E[Z^{(a)};Z^{(a)}>y\Delta]\\ 
 &\hspace{4cm}+\frac{\theta_{y\Delta}^2}{2\Delta^2}\E[(Z^{(a)})^2 \exp\{\gamma \theta_{y\Delta}  Z^{(a)}/\Delta\};Z^{(a)}\le y\Delta]
\end{align*}
with some random $\gamma\in (-\infty,1]$. We restrict ourselves to $\theta_{y\Delta}$ such that $\theta_{y\Delta}=O(1/y)$. 
Then, \eqref{moment_ex_Z2} implies 
$$
 \P(Z^{(a)}>y\Delta) + \frac{\theta_{y\Delta}}{\Delta} \E[Z^{(a)};Z^{(a)}>y\Delta] = O(y^{-1-\varepsilon})
$$
and 
\begin{align*}
 &\frac{\theta_{y\Delta}^2}{2\Delta^2}\E[(Z^{(a)})^2 \exp\{\gamma \theta_{y\Delta} Z^{(a)}\};Z^{(a)}\le y\Delta] \\
 &\hspace{4cm}= O\left(\theta_{y\Delta}^2 E[(Z^{(a)})^2;Z^{(a)}\le y\Delta]\right) 
 = O(y^{-1-\varepsilon}).
\end{align*}
This means that to find $\theta_y$ that suffices \eqref{eq_1}, it is sufficient to choose $\theta_y$ such that
$$
 1+\frac{\theta_{y\Delta} \mu^{(a)}}{\Delta} = \frac{1}{A} + O(y^{-1-\varepsilon})
$$
or
$$
 \theta_{y\Delta} = \frac{(1-A)\Delta}{A \mu^{(a)} } + O(y^{-1-\varepsilon}).
$$
Consequently, 
\beq
\label{lambda}
 \lambda_{y\Delta}= \exp\left\{\frac{(1-A)\Delta}{A\mu^{(a)} }+O(y^{-1-\varepsilon})\right\}.
\eeq
Further,
\begin{align*}
 &\mu_{y\Delta}^{(a)}(\lambda_{y\Delta}) 
 = \sum_{k=1}^y k f_{k\Delta}^{(a)} \lambda_{y\Delta}^{k-1} 
 = \frac{1}{\Delta\lambda_{y\Delta}} \E[Z^{(a)} \exp\{\theta_{y\Delta} Z^{(a)}/\Delta\}; Z^{(a)}\le y\Delta]\\
 &= \frac{1}{\Delta\lambda_{y\Delta}}\left\{\E[Z^{(a)};Z^{(a)}\le y\Delta]+ \frac{\theta_{y\Delta}}{\Delta} \E[(Z^{(a)})^2\exp\{\tilde{\gamma} \theta_{y\Delta} Z^{(a)}/\Delta\};Z^{(a)}\le y\Delta]\right\}
\end{align*}
for some random $\tilde{\gamma} \in (-\infty,1]$. 
For all $\theta_{y\Delta}=O(1/y)$ the result \eqref{moment_ex_Z2} gives
$$
 \E[(Z^{(a)})^2\exp\{\tilde{\gamma} \theta_y Z^{(a)}/\Delta\};Z^{(a)}\le y\Delta] = O(y^{1-\varepsilon})
$$
and 
$$
 \E[Z^{(a)};Z^{(a)}\le y\Delta]= \mu^{(a)} + O(y^{-\varepsilon}).
$$
Consequently,
\beq
\label{mu}
 \mu_{y\Delta}^{(a)}(\lambda_{y\Delta}) 
 = \frac{\mu^{(a)}}{\Delta\lambda_{y\Delta}} +O(y^{-\varepsilon}).
\eeq
Plugging the results from \eqref{lambda} and \eqref{mu} into the right hand side of \eqref{eq4.7}, 
we obtain by regarding $1-A=O(a)$,
\begin{align}
 &\P(M^{(a)}=y\Delta)\nonumber\\
 &=\frac{(1-A)\Delta}{A \mu^{(a)}+O(y^{-\varepsilon})}\exp\left\{-\frac{(1-A)y\Delta}{A\mu^{(a)}}+O(y^{-\varepsilon})\right\}
   + o(ay^{-\min\{1,\varepsilon\}}\ln y)\nonumber\\
  &= \frac{(1-A)\Delta}{A \mu^{(a)}+O(y^{-\varepsilon})}\exp\left\{-\frac{(1-A)y\Delta}{A\mu^{(a)}}\right\} + o(ay^{-\min\{1,\varepsilon\}}\ln y)\nonumber\\
  &= \frac{(1-A)\Delta}{A\mu^{(a)}}\exp\left\{-\frac{(1-A)y\Delta}{A\mu^{(a)}}\right\}
  + o(ay^{-\min\{1,\varepsilon\}}\ln y) +O(ay^{-\varepsilon}) \label{order_with_small_o}
\end{align}
uniformly for all $y$ such that $ay=O(1)$ as $a\to 0$.
Here, we applied Taylor's formula in the last line.
As a consequence of \eqref{tau_a_to_tau}, \eqref{conv_moment_Z} and \eqref{wald}, 
$$
 \frac{1-A}{A\mu^{(a)}} 
 = \frac{2a}{\sigma^2} +o(a)
$$
and hence, by plugging this result into \eqref{order_with_small_o}, we finally obtain
\begin{align*}
  \P(M^{(a)}=y\Delta)
  \sim \frac{2a\Delta}{\sigma^2} \exp\left\{-\frac{2ay\Delta}{\sigma^2}\right\}
\end{align*}
uniformly for all $y$ such that $y\to \infty$ and $ya=O(1)$ as $a\to 0$.


\vspace{12pt}
{\bf Acknowledgement.} I am grateful to Vitali Wachtel for useful ideas and references.


\end{document}